\documentclass{amsart}

\usepackage[utf8]{inputenc}
\usepackage[T1]{fontenc}
\usepackage{verbatim}
\usepackage{graphicx}
\usepackage{graphicx,caption2,psfrag,float,color}
\usepackage{amssymb}
\usepackage{amscd}
\usepackage{amsmath}
\usepackage{enumerate}

\usepackage{mathrsfs}

\usepackage[T1]{fontenc}

\newtheorem{theorem}{Theorem}[section]

\newtheorem{lemma}[theorem]{Lemma}

\numberwithin{equation}{subsection}
\newtheorem{definition}[theorem]{Definition}

\title{Prime Avoidance Property of $k$-th Powers of Piatetski-Shapiro Primes}
\author{Helmut Maier and Michael Th. Rassias}
\date{\today}
\address{Department of Mathematics, University of Ulm, Helmholtzstrasse 18, 89081 Ulm, Germany.}
\email{helmut.maier@uni-ulm.de}
\address{Department of Mathematics and Engineering Sciences,
Hellenic Military Academy,
16673 Vari Attikis, Greece 
\& Institute for Advanced Study, Program in Interdisciplinary Studies,
1 Einstein Dr, Princeton, NJ 08540, USA.}\email{michail.rassias@math.uzh.ch}\thanks{}
\begin{document}

 \maketitle
 
\begin{abstract} 
In previous papers the authors established the prime avoidance property of $k$-th powers of prime numbers and of prime numbers within Beatty sequences. In this paper the authors consider $k$-th powers of Piatetski-Shapiro primes.\\
  
\newline
\end{abstract}

\section{Introduction}
The term $``$prime avoidance'' first appears in the paper \cite{konyagin} of K. Ford, D. R. Heath-Brown and S. Konyagin, where they prove the existence of infinitely many ``prime-avoiding'' perfect $k$-th powers for any positive integer $k$. Their definition is as follows:

\begin{definition}\label{defn11}
Let 
\[
g_1(m):=\frac{\log m\: \log_2m\: \log_4 m}{(\log_3 m)^2}\:. \tag{1.1}
\]
Here $\log_k x:=\log(\log_{k-1} x)$. Then an integer $m$, for which $g_1(m)\geq m^{(0)}$ (with $m^{(0)}>0$ being a fixed constant) is called prime-avoiding with constant $c_0>0$  if $m+u$ is composite  for all integers $u$ satisfying $|u|\leq c_0 g_1(m)\:.$ We also say, $m$ has the prime avoidance property with constant $c_0$.
\end{definition}
For the sake of larger flexibility we also consider the following variations:

\begin{definition}\label{def12}
Let $h(m)$ be defined for $m\geq m^{(1)}$ and let $h(m)\rightarrow \infty$ for $m\rightarrow \infty$. Let $c_1>0$ be a fixed constant.
\begin{enumerate}[(i)]
\item Then an integers $m\geq m^{(1)}$ (with $m^{(1)}>0$ being a fixed constant) is called prime-avoiding with constant $c_1$ and function $h$ if $h(m)\geq m^{(1)}$ and $m+u$ is composite for all integers $u$ satisfying $|u|\leq h(m)$. We also say, that $m$ has the prime avoidance property with constant $c_1$ and function $h$.
\item $m\geq m^{(1)}$ is called prime-avoiding to the right with function $h$ and constant $c_2$ if $m+u$ is composite for all integers $u\in [0, c_2 h(m)]$. We say that $m$ has the prime-avoidance property to the right.\\
An analogous definition is to apply for ``prime-avoiding to the left'' and ``prime avoidance property to the left''.
\item The terms ``one-sided prime-avoidance property'' is used, if prime-avoidance property to the right, to the left of both hold.\\
\end{enumerate}
\end{definition}
The function $g_1$ in (1.1) has played a central role for a long time in the study of large gaps between consecutive primes. The connection between these questions can easily be seen as follows:\\
Let $m$ be prime-avoiding with function $h$ and constant $c_3>0$ and let $p_n$ be the largest prime number $\leq m$, with $p_{n+1}$ being the subsequent prime number. Then 
$$p_{n+1}-p_n\geq 2c_1h(m)\:.$$
Also the reverse $-$ large gaps imply the existence of prime-avoiding numbers can be seen easily.

\section{Large gaps between primes}

The fact that 
$$\limsup_{n\rightarrow \infty} \frac{p_{n+1}-p_n}{\log p_n}=\infty$$
was first proved by Westzynthius \cite{Westzynthius}. Erd\H os \cite{erdos} obtained that infinitely often one has:
\[
p_{n+1}-p_n > C_1\: \frac{\log p_n\:\log_2 p_n}{(\log_3 p_n)^2}  \tag{2.1}
\]
with appropriate $C_1>0$.\\
The order of magnitude of the lower bound (2.1) was slightly improved by Rankin \cite{rankin} who showed that
\[
p_{n+1}-p_n> C_2 \: \frac{\log p_n\:\log_2 p_n\:\log_4 p_n}{(\log_3 p_n)^2}= C_2 g_1(p_n)\:.
\tag{2.2}
\]
The proof of Rankin differs only slightly from that of Erd\H os. Rankin is using a better estimate for the number of smooth integers. The approach of Rankin and modifications of it have become known as the Erd\H os-Rankin method. In \cite{erdos} and \cite{rankin} these two authors construct a long interval of integers which all have greatest common divisor larger than 1 with 
$$P(x):=\prod_{p<x} p\:.$$
Their result also implies a prime-avoidance result for most integers of that interval. A modified version of this method was applied by the authors of \cite{konyagin} on the problem of prime-avoiding $k$-th powers mentioned in the introduction.\\
In \cite{mr2} the authors of the present paper extended this result by proving the existence of infinitely many prime-avoiding $k$-th powers of prime numbers. For the sake of simplicity they only treat one-sided prime avoidance. They prove the following:

\textit{There are infinitely many $n$, such that
$$p_{n+1}-p_n\geq C_3 g_1(n)$$
and the interval $[p_n, p_{n+1}]$ contains the $k$-th power of a prime number.}

Their method of proof consists in a combination of the method of \cite{konyagin} with the matrix method of the first author \cite{maier}. The matrix $\mathcal{M}$ employed in this technique is of the form
$$\mathcal{M}:=(a_{r, u})_{\substack{1\leq r\leq P(x)^{D-1}\\ u\in\mathcal{B}}}\:,$$
where $D$ is a fixed positive integer, where the \textit{rows} 
$$\mathcal{R}(r):=\{ a_{r,u}\::\: u\in \mathcal{B}\}$$
of the matrix are \textit{translates} -- in closer or wider sense -- of the \textit{base-row} $\mathcal{B}$.\\
The essential idea of the construction of the base-row $\mathcal{B}$ is the Erd\H ow-Rankin method. Here however -- as in all combinations with the matrix method -- the base-row is not completely free of numbers not coprime to $P(x)$. The columns
$$\mathcal{C}(u):=(a_{r,u})_{1\leq r\leq P(x)^{D-1}}\:,\ \ \text{with $(a, P(x))=1$}$$
are called \textit{admissible columns}.\\
The number of prime numbers in these columns, which are arithmetic progressions (or in the case of \cite{mr2} -- shifted powers of elements of arithmetic progressions), can be estimated using theorems on primes in arithmetic progressions.\\
It was a famous prize problem of Erd\H os, being open for more that 70 years, to replace the function $g_1$ in (2.2) by a function of higher order of magnitude. This problem could finally be solved in the paper \cite{ford} by K. Ford, B. J. Green, S. Konyagin and T. Tao  and independently in the paper  \cite{may2} by J. Maynard. Later all five authors improved on this result in their joint paper \cite{ford2}. They proved:
$$p_{n+1}-p_n\geq C_4\: g_2(p_n)$$
infinitely often, where
$$g_2(m):=\frac{\log m\:\log_2 m\:\log_4 m}{\log_3 m}\:.$$ 
In the paper \cite{mr2} the authors of the present paper combined the methods of the papers \cite{konyagin}, \cite{ford2} and  \cite{maier} to
obtain the following theorem.\\

\noindent\textbf{Theorem 1.1 of \cite{mr2}}\\
\textit{There is a constant $C_5>0$ and infinitely many $n$, such that
$$p_{n+1}-p_n\geq C_5\:g_2(n)$$
and the interval $[p_n,p_{n+1}]$ contains the $k$-th power of a prime.}

\noindent Using the Definition \ref{def12} this can also be phrased as follows:\\
There is a constant $C_5>0$, such that infinitely many $k$-th powers of primes are one-sided prime-avoidant with constant $\frac{1}{2} C_5$. In the paper \cite{maier_beatty} the authors investigated the prime-avoidance of the $k$-th powers of prime numbers with Beatty sequences.

\begin{definition}\label{def21}
For two fixed real numbers $\alpha$ and $\beta$, the corresponding non-homogeneous Beatty sequence is the sequence of  integers defined by
$$\mathscr{B}_{\alpha,\beta}:=(\lfloor \alpha n+\beta\rfloor)_{n=1}^{\infty}$$
($\lfloor \cdot \rfloor$ denotes the floor function: $\lfloor u \rfloor$  the largest integer $\leq u$).
\end{definition}

\begin{definition}\label{def22}
For an irrational number $\gamma$ we define its type $\tau$ by the relation
$$\tau:=\sup\{ \rho\in\mathbb{R}\::\: \lim\inf\: n^\rho \|\gamma n\|=0\}\:\ \ \text{(see \cite{maier_beatty}).}$$
\end{definition}

\noindent In the paper \cite{maier_beatty} the authors prove:

\noindent\textbf{Theorem 1.3 of \cite{maier_beatty}.}\\
\textit{Let $k\geq 2$ be an integer. Let $\alpha, \beta$ be fixed real numbers with $\alpha$ being a positive irrational and of finite type. Then there is a constant $C_6>0$, depending only on $\alpha$ and $\beta$, such that for infinitely many $n$ we have:
$$p_{n+1}-p_n\geq C_6\: g_2(n)$$
and the interval $[p_n, p_{n+1}]$ contains the $k$-th power of a prime $\tilde{p}\in\mathscr{B}_{\alpha,\beta}$.}

\noindent In this paper we deal with \textit{Piatetski-Shapiro primes}. We prove:

\begin{theorem}\label{thm21}
Let $c_4\in(1,18/17)$ be fixed, $k\in\mathbb{N}$, $k\geq 2$. Then there is a constant $C_6>0$, depending only on $k$ and $c_4$, such that for infinitely many $n$ we have 
$$p_{n+1}-p_n\geq C_6 g_2(n)$$
and the interval $[p_n, p_{n+1}]$ contains the $k$-th power of a prime $\tilde{p}=[l^{c_4}]\:.$
\end{theorem}

\noindent Using again Definition \ref{def12} we can phrase Theorem \ref{thm21} as follows:

\begin{theorem}\label{thm21'}
Let $c_4\in(1,18/17)$ be fixed, $k\in\mathbb{N}$, $k\geq 2$. Then there is a constant $C_7>0$, such that infinitely many $k$-th powers of primes of the form $\tilde{p}=[l^{c_4}]$ have the one-sided prime avoidance property with constant $C_7$.
\end{theorem}
\noindent Again one can achieve prime avoidance by a slight modification of the proof.

\section{Construction of the matrix $\mathcal{M}$}

In several papers it was crucial for the estimate of prime numbers in arithmetic progressions $\bmod q$, that $q$ was a ``good modulus''.\\
We recall the definition:
\begin{definition}\label{def31} 
Let us call an integer $q>1$ a ``good" modulus, if $L(s,\chi)\neq 0$ for all characters $\chi\bmod\:q$ and all $s=\sigma+it$ with
$$\sigma>1-\frac{C_8}{\log\left[ q(|t|+1)\right]}.$$
This definition depends on the size of $C_8>0$.
\end{definition}

\begin{lemma}\label{lem32}
There is a constant $C_9>0$, such that in terms of $C_9$, there exist arbitrarily large values of $x$, for which the modulus 
$$P(x)=\prod_{p<x} p$$
is good.
\end{lemma}
\begin{proof}
This is Lemma 3.7 of \cite{mr2}.
\end{proof}

\begin{lemma}\label{lem33}
Let $q$ be a good modulus. Then  for $x\geq q^D$ we have
$$ \sum_{1<  l \leq Q}\sum_{\chi}^* \left| \sum_{x}^{x+h} \chi(p)\log p\right|  \ll h \exp\left(-c_5\: \frac{\log x}{\log Q}\right)\:,  $$
provided $x/Q\leq h<x$ and 
$$\exp(\log^{1/2} x)\leq Q\leq x^6\:,$$
(where $\displaystyle\sum^*$ denotes the sum over all primitive characters $\bmod\: l$). The value of the constant $c_5$ and the constant $D$ depends only on the value of $C_9$ in Lemma \ref{lem32}.
\end{lemma}
\begin{proof}
This is Theorem 7 of Gallagher \cite{gallagher}.
\end{proof}

\begin{lemma}\label{lem34}
Let $x$ be sufficiently large.
\[
y=C_{10}\: x\: \frac{\log x\log_3 x}{\log_2 x}\:, \tag{3.1}
\]
$C_{10}>0$ being a sufficiently small constant. Then there is an integer $m_0$ satisfying 
\[
 1\leq m_0< P(x) \tag{3.2}
\]
and an exceptional set $V$ satisfying 
\[
\#V\ll x^{1/2+\epsilon}\:,\ \ \text{$\epsilon>0$ arbitrarily small} \tag{3.3}
\]
such that 
\[
(m_0+1)^k+u-1 \tag{3.4}
\]
is composite for $2\leq u\leq y$, unless $u\in V$.
\end{lemma}
\begin{proof}
This is Lemma 3.10 of \cite{konyagin}.
\end{proof}

The construction of $V$, which is due to Ford, Heath-Brown and Konyagin \cite{konyagin} is described in Lemma 3.5 of \cite{konyagin}. We now construct the matrix $\mathcal{M}$.

\begin{definition}\label{def35}
Choose $x$, such that $P(x)$ is a good modulus. Let $y$ satisfy (3.1), $V$ satisfy (3.3) and $m_0$ satisfy (3.2) and (3.4). We let
$$\mathcal{M}:=(a_{r,u})_{\substack{1\leq r\leq P(x)^{D-1}\\ 1\leq u\leq y}}\:,$$
where 
$$a_{r,u}:=(m_0+1+rP(x))^k+u-1$$
and the constant $D$ will be specified later.\\
For $1\leq r\leq P(x)^{D-1}$ we denote by
$$R(r):=(a_{r,u})_{1\leq u\leq y}$$
the $r$-th row of $\mathcal{M}$ and for $1\leq u\leq y$ we denote by
$$C(u):=(a_{r,u})_{1\leq r\leq P(x)^{D-1}}$$
the $u$-th column of $\mathcal{M}$.
\end{definition}

\section{Piatetski-Shapiro primes in arithmetic progressions}

\begin{definition}\label{def41}
Let $c_6>1$. Let $a$ and $d$ be coprime integers, $w\geq 1$. Then we let
$$\pi_{c_6}(w; d, a)=\#\{ p\leq w\::\: p\in\mathcal{P}^{(c_6)},\ p\:\equiv a \bmod\: d\}$$
$$\mathcal{P}^{(c_6)}=\{p\ \text{prime}\::\: p=[l^{c_6}]\ \text{for some}\ l\}\:.$$
\end{definition}

\begin{lemma}\label{lem42}
Let $a$ and $d$ be coprime integers, $d\geq 1$. For fixed $c_6\in(1, 18/17)$ we have (with $\gamma=1/c_6$):
$$\pi_{c_6}(w; d, a)= \gamma w^{\gamma-1} \pi(w; d, a)+\gamma(1-\gamma)\int_2^w u^{\gamma-2} \pi(u; d, a)du+O\left( w^{\frac{17}{39}+\frac{7\gamma}{19}+\epsilon} \right)\:.$$
\end{lemma}
\begin{proof}
This is Theorem 8 of \cite{Banks}.
\end{proof}

\section{Conclusion of the Proof}

\begin{definition}\label{def41}
We let $l(r):=m_0+1+rP(x)$,
$$\mathcal{R}_1(\mathcal{M}):=\{r\:: 1\leq r\leq P(x)^{D-1},\ l(r)\in\mathcal{P}^{(c_6)}\}\:,$$
$$\mathcal{R}_2(\mathcal{M}):=\{r\:: 1\leq r\leq P(x)^{D-1},\ r\in\mathcal{R}_1(\mathcal{M}),\ {R}(r)\ \text{contains a prime number}\}.$$
\end{definition}

We observe that each row ${R}(r)$ with $r\in\mathcal{R}_1(\mathcal{M})$ has as its first element $a_{r,1}$, the $k$-th power of the prime $l(r)\in\mathcal{P}^{(c_6)}$. We now conclude the proof of Theorem \ref{thm21} by showing that the set  $\mathcal{R}_1(\mathcal{M})\setminus \mathcal{R}_2(\mathcal{M})$ is non-empty.

\begin{lemma}\label{lem52}
We have
$$\#\mathcal{R}_1(\mathcal{M})=\frac{P(x)^{D\gamma}}{\phi(P(x))}\: \left(1+O\left( e^{-c_7D}\right)\right)\:.$$
\end{lemma}
\begin{proof}
We apply Lemma \ref{lem42} with $P(x)^D$ instead of $w$, $d=P(x)$, $a=m_0+1$ and obtain 
\begin{align*}
\#\mathcal{R}_1(\mathcal{M})&=\gamma P(x)^{D(\gamma-1)}\pi(P(x)^D, P(x), m_0+1)\\
&\ \ +\gamma(1-\gamma)\int_2^{P(x)^D} u^{\gamma-2} \pi(u, P(x), m_0+1)\: du+O\left(P(x)^{D\left(\frac{17}{39}+\frac{7\gamma}{19}+\epsilon\right)}\right)\:.\tag{5.1}
\end{align*}
We now use  Dirichlet-characters to evaluate $\pi(u, P(x), m_0+1)$. We use the familiar definition
$$\theta(u, g, r)=\sum_{\substack{p\leq u \\ p\equiv r \bmod g}}\log p$$
and obtain
\[
\theta(u, P(x), m_0+1)=\frac{1}{\phi(P(x))}\sum_{p\leq u}\log p+\frac{1}{\phi(P(x))}\sum_{\substack{\chi \bmod P(x)\\ \chi\neq \chi_0}}\sum_{p\leq u}\chi(p)\log p\:.\tag{5.2}
\]
Assume that $\chi \bmod\: r$ is induced by the primitive character $\chi^*$. Then it is well-known that
$$\sum_{p\leq u}|\chi(p)\log p-\chi^*(p)\log p|\leq \log(ru)^2\:.$$
We obtain:
\begin{align*}
&\left|\theta(u, P(x), m_0+1)-\frac{1}{\phi(P(x))}\sum_{p\leq u}\log p\right| \tag{5.3}\\  
&\leq \frac{1}{\phi(P(x))}\sum_{1<r\leq P(x)}\sum_{\chi^*\bmod\: r}^{*}\left|\sum_{p\leq u}\chi(p)\log p\right|\:,
\end{align*}
where $\displaystyle\sum^*$ denotes the summation over primitive characters.\\
From Lemma \ref{lem33} we obtain:
$$\theta(u, P(x), m_0+1)=\frac{1}{\phi(P(x))}\left(\sum_{p\leq u}\log p\right)\left(1+O\left(e^{-c_7D}\right)\right)\:.$$
Lemma \ref{lem52} now follows from (5.1), (5.2) and (5.3).
\end{proof}

We now estimate $\mathcal{R}_2(\mathcal{M})$. We have
$$|\mathcal{R}_2(\mathcal{M})|\leq \sum_{v\in V} t(v)\:,$$
where
$$t(v)=\#\{ r\::\: l(r)\in \mathcal{P}^{(c_6)}, l(r)^k+v-1\ \text{prime}  \}\:.$$
To obtain an upper estimate for $t(v)$ we replace the condition $l(r)^k+v-1$ prime by the condition 
$$(l(r)^k+v-1, Q(x))=1\:,\ \ \text{where}\ \ Q(x)=\prod_{\substack{x<p\leq P(x) \\ p\ \text{prime}}} p\:.$$
Let
\[
N(v):=\#\{r\::\: l(r)\in \mathcal{P}^{(c_6)}: (l(r)^k+v-1, Q(x))=1\}\:. \tag{5.4}
\]
By the sieve of Eratosthenes we have:
\[
N(v)=\sum_{r:l(r)\in \mathcal{P}^{(c_6)}, l(r)\leq P(x)^D}\ \  \sum_{\tau\mid (l(r)^k+v-1, Q(x))} \mu(\tau) \tag{5.5}
\]
We now apply a Fundamental Lemma in the theory of Combinatorial Sieves, replacing the M\"obius functions by functions of compact support.

\begin{lemma}\label{lem53}
Let $\kappa>0$ and $y>1$. There exist two sets of real numbers $\Lambda^+=(\lambda_d^+)$ and $\Lambda^-=(\lambda_d^-)$ depending only on $\kappa$ and $y$ with the following properties:
\begin{align*}
&\lambda_1^{\pm}=1\:,\tag{5.6}\\
&|\lambda_d^{\pm}|\leq 1\:,\ \text{if}\ 1<d<y\:,\tag{5.7}\\
& \lambda_d^{\pm}=0\:,\ \text{if}\ d\geq y\:,\ \text{and for any integer $n>1$,}\tag{5.8}\\
&\sum_{d\mid n} \lambda_d^- \leq 0 \leq  \sum_{d\mid n} \lambda_d^+ \:.\tag{5.9}
\end{align*}
Moreover, for any multiplicative function $g(d)$ with $0\leq g(p)<1$ and satisfying the dimension condition
$$\prod_{w\leq p< z}(1-g(p))^{-1}\leq \left(\frac{\log z}{\log w}\right)^\kappa \left(1+\frac{K}{\log w}\right)$$
for all $w, z$ such that $2\leq w<z\leq y$ we have
$$\sum_{d\mid P(z)}\lambda_d^{\pm} g(d)=\left(1+O\left(e^{-s}\left(1+\frac{K}{\log z}\right)^{10}\right)\right) \prod_{p<z}(1-g(p))\:.$$
\end{lemma}
\begin{proof}
This is Fundamental Lemma 6.3 in \cite{IKW}.
\end{proof}

For $\tau\mid Q(x)$ let $\rho(\tau, v)$ be the number of solutions of 
$$w^k+v-1 \equiv 0 \bmod\: \tau\:.$$
The function $\rho(\cdot, v)$ is multiplicative:
$$\rho(\tau, v)=\prod_{\tilde{p}\mid \tau} \rho(\tilde{p}, v)\:.$$
We now apply Lemma \ref{lem53} with $g$ given by
\begin{eqnarray}
\nonumber\\
g(\tilde{p})&=&\left\{ 
  \begin{array}{l l}
    \rho(\tilde{p}, v), \ \ \  \text{if $x\leq \tilde{p}\leq P(x)$} \vspace{2mm}\\ 
    0, \ \ \  \text{otherwise}\vspace{2mm}\\ 
  \end{array} \right.
\nonumber
\end{eqnarray}
and choose $\lambda_d^+$ satisfying (5.7), (5.8), (5.9), where $y=P(x)^E$, $E$ to be determined later.\\
From (5.1) and (5.2), we obtain
$$N(v)\leq \sum_{\tau\mid Q(x)} \lambda^+(\tau)\sum_{\substack{r\::\: l(r)\in \mathcal{P}^{(c_6)},\ l(r)\leq P(x)^D \\ l(r)^k+v-1\equiv 0 \bmod\: \tau}} 1$$

For $\tau\mid Q(x)$ let $\mathcal{S}(\tau, v)$ be the solution set of 
$$s^k+v-1\equiv 0\bmod\: \tau\:.$$
For $s\in \mathcal{S}(\tau, v)$ the set of congruences 
\begin{align*}
l(r)&\equiv\: m_0+1 \bmod\: P(x)\\
l(r)&\equiv\: s \bmod\: \tau
\end{align*}
coincides with the solution set of a single congruence
$$l(r)\equiv \: z(s) \bmod\: \tau P(x)\:.$$
Thus we have
\begin{align*}
&\sum_{\substack{r\::\: l(r)\in \mathcal{P}^{(c_6)},\ l(r)\leq P(x)^D \\ l(r)^k+v-1\equiv 0 \bmod\: \tau}} 1=
\sum_{s\in\mathcal{S}(\tau, v)}\ \ \sum_{\substack{p\in  \mathcal{P}^{(c_6)},\ p\leq P(x)^D\\ p\equiv\: z(s)\bmod\: \tau P(x)}} 1\\
&=\sum_{s\in\mathcal{S}(\tau, v)} \pi_{c_6}(P(x)^D, \tau P(x), z(s))\\
&=\sum_{s\in\mathcal{S}(\tau, v)} \bigg( \gamma P(x)^{D(\gamma-1)} \pi\left(P(x)^D, \tau P(x), z(s)\right)\\
&\ \ \ \ \ \ +\gamma(1-\gamma)\int_2^{P(x)^D} u^{\gamma-2} \pi(u, \tau P(x), z(s)) du \bigg)\:.
\end{align*}
We write
$$\pi(u, q, a)=\frac{\theta(u, q, a)}{\log u}\ \left(1+O\left(\frac{1}{\log u}\right)\right)$$
and make use of Dirichlet characters.\\
For the non-principal characters we again apply Lemma \ref{lem33}.\\
In the treatment of the principal character we switch from the function $\lambda^+$ to the M\"obius function. Using Lemma \ref{lem53} we obtain
\begin{align*}
&\sum_{\tau\mid Q(x)} \lambda^+(\tau)\:\frac{\rho(\tau, v)}{\phi(\tau)\phi(P(x))}P(x)^{D\gamma}\left(1+O\left(e^{-c_{11}D}\right)\right)\\
&\leq  \frac{P(x)^{D\gamma}}{\phi(P(x))}\:\left(1+O\left(e^{-s}\left(1+\frac{\kappa}{\log z}\right)^{10}\right)\right)\: \prod_{x< \tilde{p}\leq P(x)}\left(1-\frac{\rho(p, v)}{p}\right)\\
&\ll \frac{P(x)^{D\gamma}}{\phi(P(x))}\: \frac{(\log x) |x|^\epsilon}{x}\:,
\end{align*}
by Lemma 3.1 of \cite{gallagher}.\\
The proof is now finished by the estimate (3.3). \qed

\end{document}